\documentclass[11pt]{amsart}
\usepackage{latexsym,amssymb,amsmath}
\input epsf
\def\boxit#1{\vbox{\hrule height1pt\hbox{\vrule width1pt\kern3pt
  \vbox{\kern3pt#1\kern3pt}\kern3pt\vrule width1pt}\hrule height1pt}}

\usepackage{amsthm}
\usepackage{amsmath}
\usepackage{amsfonts}
\usepackage{amssymb}
\usepackage{amscd}
\usepackage{mathrsfs}
\usepackage[all]{xypic}
\usepackage{url}

\newtheoremstyle{custom}
  {3pt}
  {3pt}
  {\slshape}
  {}
  {\bfseries}
  {.}
  { }
   {}
\theoremstyle{custom}
\newtheorem{theorem}{Theorem}[subsection]

\newtheorem{proposition}[theorem]{Proposition}
\newtheorem{proposition/definition}[theorem]{Proposition/Definition}

\newtheorem{corollary}[theorem]{Corollary}
\newtheorem{conjecture}[theorem]{Conjecture}

\theoremstyle{definition}
\newtheorem{definition}[theorem]{Definition}

\newtheorem{example}[theorem]{Example}

\newtheorem{question}[theorem]{Question}

\theoremstyle{remark}
\newtheorem{remark}[theorem]{Remark}

\def\donote#1{\noindent{\bf #1\ }}

\newtheoremstyle{exercise}
  {3pt}
  {6pt}
  {}
  {}
  {\bfseries}
  {:}
  { }
   {}
\theoremstyle{exercise}
\newtheorem{exercise}[theorem]{Exercise}
\newtheoremstyle{exercises}
  {3pt}
  {6pt}
  {}
  {}
  {\bfseries}
  {:}
  {\newline}
   {}

\theoremstyle{exercise}
\newtheorem{exercises}[theorem]{Exercises}

\input epsf
\def\boxit#1{\vbox{\hrule height1pt\hbox{\vrule width1pt\kern3pt
  \vbox{\kern3pt#1\kern3pt}\kern3pt\vrule width1pt}\hrule height1pt}}

\def\BC{\mathbb C}\def\BF{\mathbb F}\def\BS{\mathbb S}

\def\BP{\mathbb P}
\def\pp#1{\mathbb P^{#1}}

\def\fb{\mathfrak b}

\def\pp#1{{\mathbb P}^{#1}}
\def\tdim{{\rm dim}}

\def\hd{,...,}
\def\ww{\wedge}
\def\upperp{{}^\perp}

\def\cS{{\mathcal S}}

\def\cW{{\mathcal W}}
\def\cO{{\mathcal O}}

\def\11{\mathbf 1}

\def\fg{{\mathfrak g}}

\def\fp{{\mathfrak p}}

\def\a{\alpha}

\def\o{\omega}

\def\s{\sigma}

\def\ot{{\mathord{ \otimes } }}
\def\op{{\mathord{\,\oplus }\,}}
\def\otc{{\mathord{\otimes\cdots\otimes}\;}}

\def\ra{{\mathord{\;\rightarrow\;}}}

\def\La#1{\Lambda^{#1}}

\def\ccdots{\circ\cdots\circ}

\def\op{\oplus}
\def\BF{\Bbb F}\def\BZ{\Bbb Z}

\def\ep{\epsilon}
\def\op{\oplus}

\def\s{\sigma}

\def\a{\alpha}

\def\FS{\mathfrak  S}

\def\ol{\overline}

\def\BP{\mathbb  P}
\def\BC{\mathbb  C}

\def\pp#1{\mathbb  P^{#1}}

\def\BS{\mathbb  S}

\def\ep{\epsilon}
\def\fp{\mathfrak  p}

\def\fg{\mathfrak  g}

\def\frp#1#2{\frac{\partial {#1}}{\partial {#2}}}

\def\hd{, \hdots ,}

\def\La#1{\Lambda^{#1}}

\def\pp#1{\mathbb  P^{#1}}

\def\ra{\rightarrow}

\def\tdeg{\operatorname{deg}}
\def\tdet{\operatorname{det}}\def\tpfaff{\operatorname{Pf}}
\def\sPf{\operatorname{sPf}}

\def\tperm{\operatorname{perm}}

\def\tdim{\operatorname{dim}}

\def\tlim{\lim}

\def\tmin{\operatorname{min}}

\def\upperp{{}^{\perp}}

\def\ww{\wedge}

\def\ctimes{\times \cdots\times}

\def\be{\begin{equation}}
\def\ene{\end{equation}}
\def\aaa{{\bold {a}}}

\def\sgn{{\rm{sgn}}}\def\tsgn{{\rm{sgn}}}

\def\frp#1{\frac{\partial}{\partial{#1}}}

\def\dual{{^\vee}}

\def\vp{{\bold V\bold P}}
\def\vnp{{\bold V\bold N\bold P}}\def\p{{\bold P}}
\def\np{{\bold N\bold P}}

\def\G{\Gamma}

\newcommand{\sign}{\operatorname{sgn}}

\def\tzeros{{\rm Zeros}}

\def\BC{\mathbb C}\def\BF{\mathbb F}\def\BS{\mathbb S}

\def\BP{\mathbb P}
\def\pp#1{\mathbb P^{#1}}

\def\fb{\mathfrak b}

\def\pp#1{{\mathbb P}^{#1}}
\def\tdim{\rm dim}
\def\hd{,...,}
\def\ww{\wedge}
\def\upperp{{}^\perp}

\def\be{\begin{equation}}
\def\ene{\end{equation}}
\def\aaa{{\bold a}}

\def\tmin{\operatorname{min}}

\def\G{\Gamma}

\def\cS{{\mathcal S}}

\def\cW{{\mathcal W}}
\def\cO{{\mathcal O}}

\def\11{\mathbf 1}

\def\BZ{\mathbb Z}

\def\FS{{\mathfrak S}}

\def\fg{{\mathfrak g}}

\def\fp{{\mathfrak p}}

\def\a{\alpha}

\def\o{\omega}

\def\s{\sigma}

\def\ep{\epsilon}

\def\ot{{\mathord{\,\otimes }\,}}
\def\op{{\mathord{\,\oplus }\,}}
\def\otc{{\mathord{\otimes\cdots\otimes}\;}}\def\ctimes{{\mathord{\times\cdots\times}\;}}

\def\ra{{\mathord{\;\rightarrow\;}}}

\def\tdim{{\rm dim}}
\def\tzeros{{\rm Zeros}}
\def\tdet{{\rm det}}\def\tlim{{\rm lim}\;}
\def\tperm{{\rm perm}}
\def\La#1{\Lambda^{#1}}

\def\vp{\bold {VP}}\def\vnp{\bold {VNP}}

\numberwithin{equation}{subsection}
 \numberwithin{theorem}{subsection}

\def\ol{\overline}

\def\p{{\bold P}}

\begin{document}
\title{$P$ versus $NP$ and geometry}
\author{J.M. Landsberg}
\begin{abstract} In this primarily expository article, I describe   geometric approaches  to variants of $P$ v.  $NP$,
present several results that illustrate the role of group actions in complexity theory,  and make a first
step towards
   geometric definitions of complexity classes. My goal is to help bring geometry and complexity theory
closer together.
\end{abstract}
 \thanks{supported by NSF grant  DMS-0805782}
\email{jml@math.tamu.edu}
\maketitle

\section{Introduction}

The purpose of this article is to explain  some of the beautiful problems   in
geometry  that arise in the study of $\p$ versus $\np$ and to motivate geometers
to work on them.
The article is divided into three
parts: (i)   {\it holographic algorithms}, where surprising reductions in complexity 
are related to the geometry of complex Hermitian symmetric spaces, in particular the variety
of pure spinors, (ii)    comparing the complexity of computing
the permanent and determinant polynomials, where local differential geometry,
geometric invariant theory
and representation theory all play a role, and (iii) first steps towards describing   geometric 
(i.e., coordinate free)
definitions of algebraic complexity classes. 
While these parts are formally unrelated, there are common themes arising
in each case, most importantly,   (possibly hidden) group actions.

 \medskip

Roughly speaking,   
a {\it problem} in complexity theory is a class of expressions to be evaluated  (e.g. count the number of four colorings of
a planar graph). An {\it instance} of a problem is a particular member
of the class (e.g. count the number of four colorings of the complete
graph with four vertices). $\p$ is the class of problems that admit
an algorithm that solves any instance of it in a number of steps
that depends polynomialy on the size of the input data. One says
that such problems \lq\lq admit a polynomial time solution\rq\rq .
$\np$ is the class of problems where a proposed solution to an instance
can be positively checked in polynomial time. The famous
Cook's hypothesis is $\p\neq\np$. 

\medskip

I will be concerned with two types of evaluations in this article,
here is the first:
For each $n$, let $V_n$ be a complex vector space   and 
  assume $\tdim (V_n)$ grows exponentially with  $n$. It is known that
the pairing 
\begin{align*}V_n\times V_n^*&\ra \BC\\
(v,\a)&\mapsto \langle \a,v\rangle
\end{align*}
 of the vector space with
its dual requires on the order of $\tdim (V_n)$ arithmetic operations
to perform. However if $V_n$ has additional structure and 
$\a,v$ are in \lq\lq special position\rq\rq\ with respect to
this structure, the pairing can be evaluated faster. A trivial
example would be if $V_n$ were equipped with a basis and $v$ 
was restricted to be a linear combination of only the first
few basis vectors. I will be concerned with  more subtle examples  
such as the following: let $V_n=\La k\BC^n$, then
inside $V_n$ are the decomposable vectors (the cone
over the Grassmannian $G(k,\BC^n)$). If $\a,v$ are decomposable, 
 equation \eqref{avpair} shows that the pairing $\langle \a,v\rangle$ can be
evaluated  in   polynomial time in $n$.
From a geometric perspective, this is one of the key
ingredients to L. Valiant's {\it holographic algorithms}
discussed in \S\ref{holoalgsecta} and \S\ref{holoalgsectb}.  For $n$ large, the codimension of the Grassmannian
is huge, so it would seem highly unlikely that any interesting
problem could have $\a,v$ so special. However small Grassmannians
are of small codimension. This leads to the second key ingredient
to  holographic algorithms.  On the geometric side, if
$[v_1]\in G(k_1,W_1)$ and $[v_2]\in G(k_2,W_2)$, then
$[v_1\ot v_2]\in G(k_1k_2,W_1\ot W_2)$. Thus if our
vectors can be thought of as being built out of 
vectors in smaller spaces, there is a much better
chance of the vectors lying in the Grassmannian. Due to the nature of 
 problems in complexity theory, this is
exactly what occurs. The third key ingredient
is that there is  some flexibility in how the small vector spaces
are equipped with the additional structure, and 
I show (Theorem \ref{preholothm}) that even for $\np$-complete problems there
is sufficient flexibility to allow everything to work up
to this point. The difficulty occurs when one tries to
glue together the small vector spaces compatibly for
both $V_n$ and $V_n^*$, although even here,   the \lq\lq only\rq\rq\  problem
that can occur is one of signs, see \S\ref{goeswrongsect}.   

\medskip

The second type of evaluation I will be concerned with
is that of sequences of (homogeneous) polynomials,
$p_n\in S^{d(n)}\BC^{v(n)}$, where the degree $d(n)$
and the number of variables $v(n)$ are required to grow
at most polynomialy with $n$. A generic such sequence
is known to require an exponential (in $n$) number of
arithmetic operations to evaluate and we
are interested in characterizing the sequences
where the evaluation can be done quickly. Again there are
sequences such as
$p_n=x_1^{d(n)}+\cdots + x_{v(n)}^{d(n)}$ where it
is trivial to see that there is a polynomial time evaluation, but
there are other, more subtle examples, such as
$\tdet_n\in S^n\BC^{n^2}$ where the fast evaluation 
occurs thanks to a group action (Gaussian elimination, see \S\ref{gausselimsect}).

From a geometer's perspective, it is more interesting to look
at the zero sets of the polynomials, to get sequences
of hypersurfaces in projective spaces. Similar to the
situation above regarding signs, if one changes the
signs in the expression of the determinant, e.g., to
all plus signs to obtain the permanant, one arrives
at a $\vnp$-hard sequence, where $\vnp$ is Valiant's algebraic
analogue of $\np$, see \S\ref{vpdefssect} for a definition.

\smallskip

\donote{Problem:} Determine geometric properties of sequences
of hypersurfaces such that their defining equations
admit  polynomial time  evaluations.

\smallskip

A very tentative step towards  resolving this problem is taken in \S\ref{geomccdefsect}. A second
problem is:

\smallskip

\donote{Problem:} Determine geometric properties of sequences
of hypersurfaces such that their defining equations
are in the class $\vnp$.

\smallskip

A first observation is that if a polynomial is
easy to evaluate, then any specialization of
it is also easy to evaluate, or in other words
the polynomial associated to any linear section
of its zero set is also easy to evaluate. This
leads to Valiant's conjecture \cite{MR564634} that the
permanent sequence $(\tperm_m)$ cannot be realized as a linear
projection of the determinant sequence $(\tdet_{n(m)})$ unless
$n$ grows faster than any polynomial (Conjecture \ref{valpermconj}). The best results on
this conjecture so far are due to T. Mignon and N. Ressayre
\cite{MR2126826} who use local differential geometry. While
the local differential geometry of the $det_n$-hypersurface
is essentially understood (see Theorem \ref{detdiffinvars}),
a major
difficulty in continuing their program is to distinguish the
local differential geometry of the $\tperm_m$-hypersurface
from that of  a generic hypersurface. Furthermore, the
determinant hypersurface is so special it may be difficult to isolate
exactly which of its properties are the key to it having a fast evaluation.
Suggestions for overcoming this second difficulty are given in \S\ref{beyondgctsect}.

\smallskip

From the geometric point of view, 
a significant esthetic improvement towards approaching Valiant's conjecture
is the {\it Geometric complexity theory} (GCT) program proposed by K. Mulmuley
and M. Sohoni in \cite{MS1,MS2}. Instead of regarding the determinant itself, one considers
its $GL_{n^2}$-orbit closure in $\BP (S^n\BC^{n^2})$ and similarly for
the permanent. The problem becomes one to compare two
algebraic varieties that are invariant under a group action. In \S\ref{gctdessect}
I briefly review the  program, summarizing from  \cite{BLMW}.
Even   with the GCT program, one still begins with the determinant and permanent,
and it might be useful to consider other sequences as well, as discussed in  \S\ref{beyondgctsect}.

The examples up to this point indicate that sequences in $\vp$ that are not  in $\vp_e$,
the sequences of polynomials having \lq\lq small\rq\rq\ expressions, see \S\ref{vpdefssect} for a precise
definition,
(and analogously for $\p$) should have some kind of symmetry, but that symmetry could be hidden. 
It would be very useful to be able to formalize the notion of \lq\lq hidden symmetry\rq\rq\ in this context.
Similarly, it would be useful to have
coordinate free  definitions of complexity classes.

\subsection*{Overview}  In \S\ref{holoalgsecta},  I describe how to convert a counting problem
to a vector space pairing.
In \S\ref{grasspinsect},  I describe how the \lq\lq big cell\rq\rq\  in the Grassmannian (resp. the
spinor variety) admits an  interpretation  as the set of vectors of minors (resp. sub-Pfaffians) in preparation for \S\ref{holoalgsectb},
where I review the reformulation of holographic algorithms of \cite{LMNholo} and point out a consequence that all problems
in $\np$ are \lq\lq nearly\rq\rq\ holographic (Theorem \ref{preholothm}). In \S\ref{cominpairsect}  
the results of \S\ref{grasspinsect} are generalized  to all cominuscule varieties.
In \S\ref{vpdefssect} I review the definitions
of Valiant's complexity classes in preparation for sections \ref{valsect}, \ref{MSsect} and  \ref{geomccdefsect}.
\S\ref{valsect} discusses  Valiant's conjecture regarding the permanent as a projection
of the determinant. There are two new results (Theorems \ref{detdiffinvars} and \ref{detprojs}) on the local differential geometry of the
hypersurface $\{ \tdet_n=0\}$ relevant for complexity.      
The {\it Geometric Complexity Theory}  program of Mulmuley and Sohoni is very briefly reviewed in  \S\ref{MSsect}.  In \S\ref{geomccdefsect} a coordinate free
definition of 
the class $\ol{\vp}_e$ is given,   where {\it joins} and {\it multiplicative joins} play a role, the latter perhaps
being defined here for the first time, and a first step is taken towards a geometric definition 
of  $\vp$, using
the idea of possibly hidden symmetries.
 Other than as noted above, the various sections can be read independently.

 The results presented in this paper 
are preliminary - the main purpose of the paper is to indicate some of the deep and beautiful connections between the 
$\p$ v.  $\np$ problem and geometry. 
For connections with other areas of mathematics, see, e.g.  
\cite{MR2334207}.

 I use the summation convention that repeated indices
appearing up and down are to be summed over their range.

\subsection*{Acknowledgments} I thank MEGA for inviting me to
give a lecture on this topic in June 2009.  This paper follows up on joint work with
P. B\"urgisser, L. Manivel and J. Weyman on the GCT program,   work  with
J. Morton and S. Norine on holographic algorithms, and reports on current work
with L. Manivel and D. The. It is a pleasure to thank
these collaborators  as well as S. Kumar, L. Valiant and J. Cai for helping
me understand the computer science literature and many useful discussions.
The AIM workshop
{\it Geometry and representation theory of tensors for computer science, statistics and other areas}   July 21-25, 2008,
was especially useful as a starting point for these conversations and I 
  gratefully thank AIM and the other participants of the workshop.
Finally I thank the anonymous referees for their useful suggestions.

\section{Holographic algorithms I: Counting problems as vector space pairings $A^*\times A\ra \BC$}\label{holoalgsecta}

For simplicity of exposition, I restrict to the complexity problem 
of counting the number of solutions
to equations $c_s$ over $\BF_2$ with variables $x_i$.  
This problem is called $\# SAT$ in the complexity literature.
(In complexity theory one usually deals with Boolean variables and clauses, which is
essentially equivalent to equations over $\BF_2$ but some care must be taken in the translation.)

It came as a shock to the complexity community
when L. Valiant \cite{Valiantaccident} showed  that a certain restricted counting problem
(affectionately called \lq\lq \#Pl-Rtw-Mon-3CNF\rq\rq\ in the complexity literature), where
  counting the number of solutions mod $2$ is already  $\#\p$ complete,
had the property that counting the number of solutions mod $7$ could be done
in polynomial time. J. Cai \cite{MR2277247} recognized Valiant's method could be formed in terms
of pairings of tensors in dual spaces, and the discussion below follows his formulation.
See \S\ref{holohis} below for more on the history and further references.

To convert a counting problem to a vector space pairing,  proceed as follows:

\smallskip

\donote{Step 1.}  To an instance of a problem construct a bipartite graph $\Gamma=(V_x,V_c,E)$ that encodes the problem. Here $V_x,V_c$ are the two sets of vertices
and $E$ is the set of edges. $V_x$ corresponds to the set of variables, $V_c$ to the set of equations, and
there is an edge $e_{is}$ joining the vertex of the  variable $x_i$ to the vertex of the equation $c_s$ iff
$x_i$ appears in $c_s$. 

\medskip

\donote{Step 2.}  Construct \lq\lq local\rq\rq\ tensors that encode the information at each vertex. To do this first
associate to each edge $e_{is}$ a vector space $A_{is}=\BC^2$ with basis $a_{is|0}, a_{ is|1}$ and dual basis
$\a_{is|0}, \a_{is|1}$ of $A_{is}^*$. Next, to each variable $x_i$ associate the vector space
$$
A_i:=\bigotimes_{\{s\mid e_{is}\in E\}}A_{is}
$$
and the tensor 
\be\label{gitensor}g_i:=\ot_{\{s\mid e_{is}\in E\}} a_{is|0}+\ot_{\{s\mid e_{is}\in E\}} a_{is|1}\in A_i
\ene
 which will encode that $x_i$ should be consistently 
assigned either $0$ or $1$ each time it appears.
Now to each equation $c_s$ we associate a tensor in $A_s^*:=\ot_{\{i\mid e_{is}\in E\}} A_{is}^*$ that encodes that $c_s$ is satisfied.
For example, say $x_i,x_j,x_k$ appear in $c_s$ and that
$$
c_s(x_i,x_j,x_k)=x_ix_j+x_ix_k+x_jx_k+ x_i+x_j+x_k+1
$$
which is satisfied over $\BF_2$ as long as the variables $x_i,x_j,x_k$ are not all $0$ or all $1$.
(This equation is   called {\it 3NAE} in the computer science literature.)
More generally, say $c_s$ has $x_{i_1}\hd x_{i_{d_s}}$ appearing and $c_s$ is $d_sNAE$, then
one associates the tensor
\be\label{sisnae}
r_s:= \sum_{(\ep_1\hd \ep_{d_s})\neq (0\hd 0),(1\hd 1)}
\a_{i_1,{s }|\ep_1}\otc \a_{i_{d_s},{s}|\ep_{d_s}}.
\ene

\medskip

\donote{Step 3.}  Tensor all the local tensors from $V_x$ (resp. $V_c$) together to get two tensors in dual vector spaces
with the property that their pairing counts the number of solutions. That is,
consider $G:=\ot_ig_i$ and $R:=\ot_sr_s$ respectively elements of the vector spaces $A:=\ot_{e}A_{e}$ and $A^*:=\ot_{e}A^*_{e}$.
Then, the pairing $\langle G,R\rangle$ counts the number of solutions.

\begin{remark}
Up until now I could have just taken each $A_{is}=\BZ_2$. The reason for complex numbers was to allow a larger group
action. This group action will destroy the local structure but leave the global structure unchanged. Valiant's inspiration
for doing this was quantum mechanics, where particles are replaced by wave functions.
\end{remark}

So far we have replaced our original counting problem with the problem of computing a pairing
$A\times A^*\ra \BC$ where the dimension of $A$ is exponential in the size of the input data. 
If we had arbitrary vectors, then there is no way to perform this pairing in a number of steps that
is polynomial in the size of the original data. We saw that if one is   lucky, 
the pairing can be computed quickly. In the next section I describe the geometry
underlying \lq\lq getting lucky\rq\rq\ and in the following section discussion how  to make local changes
of bases that simultaneously put  each $g_i$ and $r_s$ into spinor varieties.

\section{Detour: Grassmannians and Spinor varieties}\label{grasspinsect}

Mathematicians are used to viewing the Grassmannian as the variety parametrizing linear
subspaces of a vector space, and the spinor variety as parametrizing isotropic subspaces.
However in statistics, the 
\lq\lq big cell\rq\rq\ inside arises as the space parametrizing the set of
minors of matrices (resp. Pfaffians of skew-symmetric matrices). We show how
these second descriptions lead to the fast algorithms mentioned in the introduction.

\subsection{The Grassmannian as a variety parametrizing minors
of matrices}

Let $W$ be a vector space and let $G(k,W)$ denote
the Grassmannian of $k$-planes through the origin in $W$.
Assume WLOG that $k\leq \tdim W-k$. The
Pl\"ucker embedding $G(k,W)\subset \BP (\La k W)$ is
obtained by, given a $k$-plane $E$,  taking a basis $e_1\hd e_k$ of $E$ and
sending $E$ to  the point $[e_1\ww\cdots\ww e_k]\in \BP (\La k W)$.
The cone over the Grassmannian,   $\hat G(k,W)\subset \La k W$ is thus the set of
 $v\in \La k W$, such that there exist
$w_1\hd w_k\in W$ with  $v=w_1\ww\cdots \ww w_k$.
 
\medskip

The Grassmannian
$G(k,W)$ admits a local parametrization as follows:
Write $W=E\op F$ where $\tdim E=k$.
Let $1\leq i,j\leq k$, $ 1\leq s,t\leq n-k$, 
fix  bases $e_1\hd e_k$ of $E$ with dual basis $e^1\hd e^k$
of $E^*$, and 
$f_1\hd f_{n-k}$ of $F$ with dual basis $f^1\hd f^{n-k}$.
Say $E=[v_0]$,  $v_0\in\hat G(k,W)$ and we want to locally parametrize
$G(k,W)$ around $[v_0]$.  Choose our basis such that
$e_j=w_j$ in the description of $v$ above. Let $x^s_j$
be linear coordinates on $E^*\ot F\simeq T_{[v_0]}G(k,W)$.
The  local parametrization about $E=[v(0)]$ is   
$$
[v(x^s_i)]=[(e_1+x^s_1e_s)\ww\cdots \ww (e_k+x^s_ke_s)].
$$
In what follows I will also need to work with $G(k,W^*)$,
In our dual bases,   a local parametrization about $E^*=
\langle e^1\hd e^k\rangle=[\a(0)]$ is
$$
[\a(y^s_j)]= [(e^{ 1}+ y^{ 1}_se^s)\ww\cdots \ww (e^{k}+ y^{k}_se^s)].
$$

I next explain how to interpret   the open subset of
$G(k,W)$ described above  as the vector of minors for $E^*\ot F$.

  For vector spaces $E,F$, $\La k(E\op F)$ has the following decomposition as
a $GL(E)\times GL(F)$ module:
\begin{align*}
\La k(E\op F)=&
(\La kE\ot \La 0 F)\op (\La{k-1}E\ot \La 1F)\op (\La{k-2}E\ot \La 2 F)\\
&\op\cdots \op (\La 1E\ot \La{k-1}F)\op (\La 0 E\ot \La kF)
\end{align*}
Assume we have a volume
form on $E$ so we may identify $\La sE\simeq \La {k-s}E^*$.
We have the $SL(E)\times GL(F)$ decomposition:
\begin{align*}
\La k(E\op F)=&
(\La 0E^*\ot \La 0 F)\op (\La{1}E^*\ot \La 1F)\op (\La{2}E^*\ot \La 2 F)\\
&
\op\cdots \op (\La {k-1}E^*\ot \La{k-1}F)\op (\La kE^*\ot \La kF)
\end{align*}

Recall that $\La sE^*\ot \La sF\subset S^s(E^*\ot F)$ 
   has the geometric interpretation as
the space of $s\times s$ minors on $E\ot F^*$, i.e., with
any choices of bases,    write an element $f$ of
$E\ot F^*$ as a matrix, then a basis of 
$\La sE^*\ot \La sF$ evaluated on $f$ will give the set of $s\times s$ minors of $f$.

To see these minors  explicitly, note that
the bases of $E^*,F$ induce bases of the exterior powers.
Expanding out $v$ above in such bases, (recall that the
summation convention is in use)
\begin{align*}
v(x^s_i)=&
e_1\ww\cdots \ww e_k\\
&
+x^s_i e_1\ww\cdots \ww e_{i-1}\ww e_s\ww e_{i+1}\ww\cdots\ww e_k\\
&
+(x^s_ix^t_j-x^s_js^t_i)
e_1\ww\cdots \ww e_{i-1}\ww e_s\ww e_{i+1}\ww
\cdots \ww e_{j-1}\ww e_t\ww e_{j+1}\cdots\ww e_k\\
&+\cdots
\end{align*}
i.e., writing $v$ as a row vector in the induced basis:
$$v
=(1,x^s_i,x^s_ix^t_j-x^t_ix^s_j,\hdots)=
(1,\Delta_{i,s}(x)\hd \Delta_{I,S}(x)\hd )
$$
where we use the notation $I=(i_1\hd i_p)$
$S=(s_1\hd s_p)$ and $\Delta_{I,S}(x)$ denotes
the corresponding $p\times p$ minor of $x$.
Similarly 
$\a= (1,y^j_s,y^j_sy^i_t-y^i_sy^j_t\hd)$.

Fix  bases so  $x,y$ are
$k\times (n-k)$ matrices. I claim
\be\label{avpair}
\langle \a,v\rangle =
\tdet(I_E+{}^txy)
\ene
because  the characteristic polynomial of a product of
a $k\times \ell$ matrix ${}^tx$ with an $\ell\times k$ matrix $y$ is:
\be\label{charpolyprod}
charpoly({}^txy)(t)=\tdet(Id_E+t{}^txy)=\sum_{I,S}\Delta_{I,S}(x)\Delta_{S,I}(y)t^{|I|}.
\ene

While \eqref{charpolyprod} is no doubt classical, I include
a proof as I didn't find one in the literature.

For a linear map $f: A\ra A$, recall the induced
linear maps $f^{\ww k}: \La k A\ra \La k A$, where, if one
chooses a basis of $A$ and represents $f$ by a matrix,
then the entries of the  matrix representing $f^{\ww k}$ in the induced
basis on $\La kA$ will be the $k\times k$ minors of the matrix
of $f$. In particular, if $\tdim A=\aaa$, then, $f^{\ww \aaa}$ is multiplication
by a scalar which is $\tdet ( f)$.

Recall the decomposition:
$$
End(E^*\op F)=(E^*\op F)\ot(E^*\op F)^*= (E^*\ot F)\op (E^*\ot E)\op (F\ot F^*)\op (F^*\ot E).
$$
To each $x\in E^*\ot F$, $y\in E\ot F^*$, associate
the   element   
\be\label{star5}
-x+Id_E+Id_F+y\in End(E^*\op F).
\ene

Note that
$$\tdet \begin{pmatrix} I_E & -{}^tx\\ y&I_{F}\end{pmatrix}=
\tdet(I_E+{}^txy).
$$
Consider
\begin{align*}
&(-x+Id_E+Id_F+y)^{\ww n}= 
(Id_E)^{\ww k}\ww (Id_F)^{\ww (n-k)}
+(Id_E)^{\ww k-1}\ww (Id_F)^{\ww (n-k -1)}\ww (-x)\ww y
\\
&
+(Id_E)^{\ww (k-2)}\ww (Id_F)^{\ww (n-k -2)}\ww (-x)^{\ww 2}\ww y^{\ww 2}
+\cdots + (Id_F)^{\ww (n-2k)}\ww (-x)^{\ww k}\ww y^{\ww k}
\end{align*}
Let  
$$
e^1\ww  \cdots \ww e^k \ww f_1\ww\cdots \ww f_{n-k}\in
\La n(E^*\ot F)
$$
be a volume form.
All that remains to check is that when we re-order our
terms that the signs work out correctly, which is left to the reader.

\subsection{Spinor varieties}
For the interpretation of  spinor varieties as maximal isotropic
subspaces on a quadric, see any of \cite{MR1636473,MR1045637,LM0}. Here I simply define
the spinor variety as the Zariski closure of the set of vectors of sub-Pfaffians 
of a skew-symmetric matrix with variables as entries.
See \cite{LMsel} for the connection with the classical definition.

For $x\in \La 2\BC^{2n}$, the Pfaffian  $\tpfaff(x)\in \BC$ is defined by
$x^{\ww n}=\tpfaff(x) n! \Omega$, where $\Omega\in \La{2n}\BC^{2n}$
is a volume form  - it is a square root of $\tdet(x)$.

Let $E$ be an $n$-dimensional vector space equipped with a volume form.
Define
$(\hat\BS_+)^0$ to be the image of the map 
\begin{align*}
\La 2 E&\ra \La{even}E=:\cS_+\\
x&\mapsto 
v=(1,x^i_j\hd  \tpfaff_I(x)\hd)=:\sPf (x)
\end{align*}
as $|I|$ varies over the even numbers from $0$ to $\llcorner\frac n2\lrcorner$.
The space of sub-Pfaffians of size $2p$ is parametrized by
$\La{2p}E$. If $n$ is even,
$\cS_+$ is self dual, and if $n$ is odd, its dual is $\cS_-:=\La{odd}E$
because   $E$ is equipped with a volume form, so 
$\La {2p}E^*=\La{n-2p}E$.

Recall the decomposition
$$
\La 2(E\op E^*)=\La 2E\op E\ot E^*\op \La 2 E^*.
$$
 
Consider $x+Id_E+y\in \La 2(E\op E^*)$. Observe
that
$$
 (x+Id_E+y)^{\ww n}=
\sum_{j=0}^n (Id_E)^{\ww (n-j)} \ww x^{\ww j}\ww y^{\ww j}\in \La{2n}(E\op E^*)
$$
Let $\Omega=e_1\ww e^1\ww e_2\ww e^2\ww\cdots\ww e_n\ww e^n \in \La{2n}(E\op E^*)$ be 
a volume form.
The coefficient of the $j$-th term is the sum  
$$
\sum_{|I|=2j}\tsgn(I)\tpfaff_I(x)\tpfaff_I(y).
$$
where for an even set $I \subseteq [n]$, define $\sigma(I)=\sum_{i \in I}i$, and define $\sgn(I)=(-1)^{\sigma(I)+|I|/2}$.
 Put more invariantly,
the $j$-th term is the pairing
$$
\langle y^{\ww j},x^{\ww j}\rangle .
$$
   For a matrix $z$
define a matrix $\tilde{z}$ by setting $\tilde{z}^i_j = (-1)^{i+j+1}z^i_j$.
Let $z$ be an $n \times n$ skew-symmetric matrix. Then for every even $I \subseteq [n]$,
\[
\tpfaff_I(\tilde{z}) = \sgn(I) \tpfaff_I(z).
\]
  For $|I|=2p$, $p=1, \dots, \lfloor \frac{n}{2} \rfloor$,
$$
\tpfaff_I(\tilde{z}) = (-1)^{i_1 + i_2 + 1} \cdots (-1)^{i_{2p-1} + i_{2p} + 1}\tpfaff_{I}(z)=\sgn(I)\tpfaff_{I}(z).
$$
Thus:

\begin{theorem}\cite{LMNholo} \label{thm:tildesumPfaff}
Let $z,y$ be skew-symmetric $n \times n$ matrices. Then
$$\langle \sPf (z), \sPf\dual(y) \rangle = \tpfaff(\tilde{z}+y).
$$
In particular, when $n$ is even,  the pairing $\cS_+\times \cS_+ \ra \BC$ restricted
to $(\hat\BS_+)^0\times (\hat\BS_+)^0\ra \BC$ can be computed in polynomial time.
When $n$ is odd,  the pairing $\cS_+\times \cS_- \ra \BC$ restricted
to $(\hat\BS_+)^0\times (\hat\BS_-)^0\ra \BC$ can be computed in polynomial time.
\end{theorem}

The first few spinor varieties are classical varieties in disguise (corresponding
to coincidences of Lie groups in the first two cases and triality in the third):
\begin{align*}
\BS_2&=\pp 2\subset\pp 2\\
\BS_3&=\pp 3\subset \pp 3\\
\BS_4&=Q^6\subset\pp 7
\end{align*}
In particular, although the codimension grows very quickly, it is small
in these cases. The next case
$\BS_5\subset \pp{15}$ is not isomorphic to any classical homogeneous variety.

\section{Holographic algorithms II: Computing the vector space pairing in  polynomial time}
\label{holoalgsectb} 

\subsection{The $SL_2\BC$ action} To try to move both $G,R$ to special position so that the pairing can 
be evaluated quickly, identify all the $A_e$ with a single $\BC^2$, and allow $SL_2\BC$ to act. This action is very cheap, and
of course if we have it act simultaneously on $A$ and $A^*$, the pairing $\langle G,R\rangle$ will be unchanged.
This step cannot always be carried out, otherwise Valiant would have proved
$\p=\np$.  

To illustrate, we now restrict to $\# 3SAT-NAE$, which is still $\np$-hard.

The   tensor $g_i$ corresponding to a variable vertex $x_i$ is
\eqref{gitensor}. The   tensor  corresponding to a NAE clause $r_s$ is
\eqref{sisnae} and   $d_s=3$ for all $s$.
Let 
$$T=\begin{pmatrix} 1&1\\ 1&-1\end{pmatrix}
$$ 
be the basis change, the same in each $A_e$, sending $a_{e|0} \mapsto a_{e|0}+a_{e|1}$ and $a_{e|1} \mapsto a_{e|0}-a_{e|1}$
which induces the basis change $\a_{e|0} \mapsto \frac{1}{2}(\a_{e|0}+\a_{e|1})$ and $\a_{e|1} \mapsto \frac{1}{2}(\a_{e|0}-\a_{e|1})$ in $A^*_e$.   Applying $T$, gives
\[
T 
(a_{i,s_{i_1}|0}\otc a_{i,s_{i_{d_i}}|0}+ a_{i,s_{i_1}|1}\otc a_{i,s_{i_{d_i}}|1})
=2\sum_{\{(\ep_1\hd \ep_{d_i}) \mid \sum \ep_\ell=0 \,(\text{mod } 2) \}}
a_{i,s_{i_1}|\ep_1}\otc a_{i,s_{i_{d_i}}|\ep_{d_i}}.
\]
Similarly
\begin{align*}
&T  \left ( \sum_{(\ep_1,\ep_2,\ep_3)\neq (0,0, 0),(1,1, 1)}
\a_{i_1,{s }|\ep_1}\ot  \a_{i_{2},{s}|\ep_{2}} \ot \a_{i_{3},{s}|\ep_{3}}
\right )\\
&=
6\a_{i_1,{s}|0}\ot  \a_{i_{2},{s}|0} \ot \a_{i_{3},{s}|0}
-2(\a_{i_1,{s}|0}\ot  \a_{i_{2},{s}|1} \ot \a_{i_{3},{s}|1} +
\a_{i_1,{s}|1}\ot  \a_{i_{2},{s}|0} \ot \a_{i_{3},{s}|1} +
\a_{i_1,{s}|1}\ot  \a_{i_{2},{s}|1} \ot \a_{i_{3},{s}|0}
)
\end{align*}
 
After this change of basis $g_i\in \BS_{\#\{s\mid e_{is}\in E\}}$ and $r_s\in \BS_{4}$ for all $i,s$!
 
\subsection{$\np$, in fact $\# \p$ is pre-holographic}

\begin{definition}Let $P$ be a counting problem. We will say that $P$ is {\it pre-holographic} if it
admits a formulation such that the vectors $g_i$, $r_s$ are all simultaneously representable as vectors of sub-Pfaffians.
\end{definition}

The following was proved (although not stated) in \cite{LMNholo}:

\begin{theorem}\label{preholothm} Any problem in $\np$, in fact in $\# \p$,  is pre-holographic.
\end{theorem}

\begin{proof} To prove the theorem it suffices to exhibit one $\#\p$ complete problem that
is pre-holographic. Counting the number of solutions to  $\# 3SAT-NAE$ is one such.
\end{proof}

\subsection{What goes wrong}\label{goeswrongsect}
While for $\# 3SAT-NAE$ it is 
always possible to give $V$ and $V^*$  structures of
the spin  representations $\cS_+$ and $\cS_+^*$, so
that $[G]\in \BP V$ and $[R]\in \BP V^*$ both lie in spinor
varieties, these structures {\it may not be compatible}!
What goes wrong is that the ordering of pairs of indices $(i,s)$
that is good for $V$ may not be good for $V^*$. The \lq\lq only\rq\rq\  thing
that can go wrong are the signs of the sub-Pfaffians, see \cite{LMNholo} for details.

In \cite{LMNholo} we determine sufficient conditions for there to be
a good ordering of indices and show that if the bipartite graph $\Gamma$
was planar, then these sufficient conditions hold.

\subsection{History}\label{holohis} In Valiant's original formulation of holographic algorithms
(see \cite{MR2120307,MR1932906,ValiantSimulatingQCiPT,ValiantFOCS2004,MR2184617,MR2386281}),
the   step of forming $\Gamma$ is the same, but then Valiant replaced
the vertices of $\Gamma$ with weighted graph fragments to get
a new weighted graph $\Gamma'$ in such a way
that the number of (weighted) perfect matchings of $\Gamma'$ equals
the answer to the counting problem. Then, if $\Gamma'$ is planar, one can
appeal to the famous FKT algorithm \cite{MR0253689,MR0136398} to compute the number of
weighted perfect matchings in polynomial time. Valiant also found certain
algebraic identities that were necessary conditions for the existence of
such graph fragments.

Cai \cite{MR2305569, MR2277247,MR2354219,MR2402465,MR2424719,MR2362482,MR2417594} recognized that Valiant's procedure could be reformulated
as a pairing of tensors as in steps two and three, and that the condition
on the vertices was that the local tensors $g_i$ $r_s$ could, possibly
after a change of basis, be realized as a vector of sub-Pfaffians. In Cai's formulation
one still appeals to the existence of $\Gamma'$ and  the FKT algorithm in the last step.

\section{Exponential pairings in polynomial time}\label{cominpairsect}

In this section we show that the same phenomenon that we observed above for Grassmannians and spinor varieties
holds for all {\it cominuscule} varieties, the homogeneous varieties that can be given the structure
of a compact Hermitian symmetric space.

\subsection{Cominscule varieties}

\begin{theorem}\label{lowbvers} Let $V$ be a vector space of dimension $\binom nk$,
$2^{n-1}$, $\binom{2n}k-\binom{2n}{k-2}$, $p^n$, or $\binom{n+p-1}n$. In each case there
are explicit systems of degree two polynomial equations on $V,V^*$, such that
if $\a\in V^*$ and $v\in V$ satisfy these equations, the pairing
$\langle \a, v\rangle$, which na\"ively requires $ O(\tdim V)$
arithmetic operations, can be computed in $ O(n^4)$ operations.
\end{theorem}
 
 Theorem \ref{lowbvers} is an immediate consequence of:

\begin{theorem}\label{highbvers} Let $V=V(n)$ be a  cominuscule $G=G(n)$-module  with $G/P\subset \BP V$ the closed orbit
and $G/P'\subset \BP V^*$ the corresponding closed orbit in the dual space. Here $n$ is the rank of $G$. Then
the pairing $V\times V^*\ra \BC$ restricted
to $\hat G/P\times \hat G/P'$ can be computed in $ O(n^4)$ arithmetic operations without divisions.
\end{theorem}

  The non-trivial cases are (where for notational convenience we use the
  rank of $G$ plus one in the $A_{n-1}=SL_n$-case):
$$ 
\begin{array}{c|c|c|c|c}
  V&\tdim V& G&  G/P&\fg/\fp  \\
\hline \\
\La k W &\binom nk &SL(W)=SL_n& G(k,W)&E^*\ot F\\
\cS_+ &2^{n-1}&D_n=Spin_{2n}& \BS_+&\La 2 E\\
\La {\langle n\rangle} W&\binom {2n}n-\binom{2n}{n-2}&
Sp(2n,\BC)=Sp(W,\o )& G_{Lag}(n,2n)&S^2E\\
E_1\otc E_n & p^n & SL(E_1)\ctimes SL(E_n) & Seg(\BP E_1\ctimes \BP E_n) & \oplus_jE_j' \\
S^nE & \binom{n+p-1}n & SL(E) & v_n(\BP E) & E'\circ \ell^{n-1} 
\end{array}
$$

Explanations of $V$: $W$ is a vector space of dimension $n$ in the first case, $2n$ in the third, 
$\cS_+$ is the (positive) half-spin
representation of $Spin_{2n}$,
$\La {\langle n\rangle} W= \La n W/(\La{n-2}W\ww\o)$ where $\o\in \La 2 W$ is a symplectic form.
$E, E_j$ are vector spaces of dimension $p$ in the last two cases.

Explanations of $G/P$:  $G(k,W)$ denotes the Grassmannian of $k$-planes in its Plucker embedding,  $\BS_+$ the \lq\lq pure spinors\rq\rq\ or {\it spinor variety},  
$G_{Lag}(n,2n)$ denotes the Lagrangian Grassmannian of $n$-planes isotropic for the
symplectic form $\o\in \La 2\BC^{2n}$, $Seg(\BP E_1\ctimes \BP E_n)$ denotes the Segre product, the
projectivization of the set of decomposable tensors in $ E_1\otc E_n $ and
$v_n(\BP E)$ denotes the Veronese variety of the projectivization of  homogeneous polynomials of
degree $n$ on $E^*$ that are $n$-th powers of a linear form.

Explanations of $\fg/\fp$: $\fg,\fp$ are the Lie algebras of $G,P$.
Let $G_0$ denote the Levi-factor of $P$.   $G_0$ is respectively $S(GL(E)\times GL(F))$,
$GL(E)$, $GL(E)$, $GL(E_1')\ctimes GL(E_n')$, $GL(E')$. As a $G_0$-module,  $\fg/\fp$ is the tangent space to 
$G/P$   at the point of $G/P$ corresponding to 
$Id\in G$. I have written $F=W/E$.
Fix vectors $e_j\in E_j$, $e\in E$ and let $\ell_j,\ell$ respectively denote
the lines they span, then $E_j'=\ell_1\otc \ell_{j-1}\ot E_j/\ell_j\ot \ell_{j+1}\otc \ell_n$ and $E'=E/\ell$.

 In each case $\fg/\fp$ is a space of endomorphisms
and $V$  as a  $G_0$-module is the sum of the spaces of all minors (of all sizes)
of $\fg/\fp$, except in the spinor case, where one takes all sub-Pfaffians.
$\oplus_jS_{2^j}E=\oplus_jS_{2\cdots 2}E$ denotes the irreducible $GL(E)$-submodule of $\La jE\ot \La j E$ giving minors on $S^2E\subset E\ot E$.


It remains to prove the cases of the Lagrangian Grassmannian, the Segre
and the Veronese.
 
\subsection{Lagrangian Grassmannian case}
The Lagrangian Grassmannian  $G_{Lag}(n,2n)\in \BP \La {\langle n\rangle}W$ is
a linear section of $G(n,2n)\subset \BP \La {  n }W$. 
Here 
$\La {\langle n\rangle}W= \La n W/(\La{n-2}W\ww\o)= W_{\o_n}^{Sp(2n,W)}$
and the quotient may be viewed as the complement to
$\La{n-2}W\ww\o\subset \La nW$ to obtain the linear section.

The   interpretation of  an open subset  (the \lq\lq big cell\rq\rq ) of $G_{Lag}(n,W)$ is
as the set of vectors of (non-redundant) minors of symmetric
matrices. The symplectic form enables the identification
of $W/E\simeq    E^*$ and the linear subspace of  
$$E^*\ot E^* 
= \La 2 E^* \op S^2E^*
$$
corresponding to the tangent space 
is just  $S^2E^*$.
See \cite{LM0} for details. 

The subspace of $\La jE^*\ot \La jE^*$ giving rise to a non-redundant set of
minors  corresponds to the sub-module
$S_{2^j}E^*\subset \La jE^*\ot \La jE^*$.

For the Lagrangian Grassmannian case it suffices in \eqref{star5} to take
$$
-x+Id_E+y\in  S^2(E\op E^*)=S^2 E\op E\ot E^*\ot S^2E^*.
$$

\subsection{Segre and Veronese cases}
The Segre is parametrized by a map 
$\phi$
$$
(x^j_s)\mapsto (a_0^1+x^j_1a^1_j)\ot (a_0^2+x^j_2a^2_j)\otc (a_0^n+x^j_na^n_j)
=(1, x^j_{s_1},x^j_sx^k_{s_2},\cdots x^{1}_{s_1}\cdots x^p_{s_p}),
$$
where in each term
$\ s_1<\cdots <s_q$. Let $\phi\dual$ denote the map to the dual Segre.

If $\a=\phi(x)$, $v=\phi\dual (y)$ then
$$
\langle \a,v\rangle=\sum_{I,S}x^I_Sy^I_S
$$
where $I=(i_1\hd i_q)$, $i_1\leq \cdots \leq i_q$, $1\leq q\leq p$, and
$S=(s_1\hd s_r)$, $s_1< \cdots < s_r$, $1\leq r\leq n$.
Here  :
$$\langle \a,v\rangle =
\tdet\begin{pmatrix} I_n &\begin{matrix} -x^1_1 & & \\ & \ddots & \\ & & -x^1_n \end{matrix} 
&\begin{matrix} -x^2_1 & & \\ & \ddots & \\ & & -x^2_n \end{matrix} & \cdots & 
&\begin{matrix} -x^p_1 & & \\ & \ddots & \\ & & -x^p_n \end{matrix}  
\\
\begin{matrix} y^1_1 & & \\ & \ddots & \\ & & y^1_n \end{matrix} & I_n & & &
\\
\\
\begin{matrix} y^2_1 & & \\ & \ddots & \\ & & y^2_n \end{matrix} &  & I_n & &
\\
\vdots &  &   & \ddots  &
\\
\begin{matrix} y^p_1 & & \\ & \ddots & \\ & & y^p_n \end{matrix} &  &  & &\  \ I_n
\end{pmatrix}.
$$

The Veronese is parametrized by $(x^j)\mapsto (a_0+ x^ja_j)^p$ and the
same matrix as above works replacing $x^j_s$ with $x^j$ for all $s$ and similarly
for $y$.

\section{Definitions of $\vp$, $\vnp$ and $\vp_e$}\label{vpdefssect}

 In the discussion above, the problem presented was far removed from geometry, and it was only after
significant work that geometric objects appeared. L. Valiant \cite{MR564634} has proposed algebraic analogs
of the complexity classes $\p$ and $\np$ in terms of sequences of polynomials. 
Such classes should be closer to geometry, however, 
the properties of the resulting sequences of hypersurfaces relevant for geometry have
yet to be determined. In this expository section, I briefly review the relevant definitions.

\subsection{$\vp_e$}  
An elementary measure of the complexity of 
a (homogeneous) polynomial $p$ is as follows:  given an expression for $p$,
  count the total number of additions plus multiplications
present in the expression, and then   take the minimum over all
possible expressions.

\begin{example}\label{xpyn}
$$
p_n(x,y)=x^n+nx^{n-1}y + \binom n2x^{n-2}y^2+\binom n3x^{n-3}y^3+\cdots +y^n
$$
This expression for $p_n$ involves $n(n+1)$ multiplications and $n$ additions,
but  one can also write
$$
p_n(x,y)=(x+y)^n
$$
which requires  $n$ multiplications and one addition to evaluate.
\end{example}

\begin{definition}\index{arithmetic circuit}
An {\it arithmetic circuit} $C$ is a finite, acyclic, directed graph with vertices of
in-degree $0$ or $2$ and exactly one vertex of out degree $0$.
In degree $0$, inputs are labelled by elements of $\BC\cup\{ x_1\hd x_n\}$
and in degree $2$, vertices are called {\it computation gates} and labelled
with $+$ or $*$. The {\it size} of $C$ is the number of vertices.
From a circuit $C$, one can construct a polynomial $p_C$ in the variables $x_1\hd x_n$.
\end{definition}

If $C$ is a tree (i.e., all out degrees are at most one), then
the size of $C$ equals the number of $+$'s and $*$'s
used in  the formula constructed
from $C$. 

\begin{definition}\label{exprsize}  For $f\in S^d\BC^m$, the
{\it expression size} $E(f)$ is the smallest size of a tree circuit  that computes $f$. 
Define the class $\vp_e$ to be the set of sequences $(p_n)$ such that
there exists a sequence $(C_n)$ of tree circuits, with
the size of $C_n$ bounded by a polynomial in $n$,  such that $C_n$
computes $p_n$.
\end{definition}
 
It turns out that expression size is too na\"ive a measurement of complexity,
as consider  Example \ref{xpyn}, we could first compute
$z=x+y$, then $w=z^2$, then $w^2$ etc... until the exponent is close
to $n$, for a significant savings in computation when $n$ is large.

\subsection{$\vp$, $\vp_{ws}$ and closures}
Circuits more general than trees
allow one to use 
 the results of previous calculation
and gives rise to the class $\vp$:

\begin{definition}\label{defvp}\index{VP}
The class $\vp$ is the set of sequences $(p_n)$ of polynomials of degree $d(n)$ in $v(n)$
variables   where $d(n),v(n)$ are
bounded by   polynomials in $n$ and such that
there exists a sequence of circuits $(C_n)$ of polynomialy bounded size  
such that $C_n$ computes $p_n$.  
\end{definition}

A polynomial $p(y_1\hd y_m)$ is a {\it projection} of $q(x_1\hd x_n)$ if
we can set $x_i=a^s_iy_s+c_i$ for constants $a^s_i,c_i$ to obtain
$p(y_1\hd y_m)=q(a^s_1y_s+c_1\hd a^s_ny_s+c_n)$. Geometrically, if we homogenize the
polynomials by adding   variables $y_0,x_0$, we can study the zero sets in projective
space. Then $p$ is a projection of $q$ iff $\tzeros(p)\subset\BC\pp m$ is a linear
section of $\tzeros (q)\subset \BC\pp n$. This is because if we consider
a projection map $V\ra V/W$, then $(V/W)^*\simeq W\upperp\subset V^*$.

\begin{definition} A sequence $(p_n)$ is {\it hard} for a complexity class $\bold C$
defined by sequences of polynomials,
if for all sequences $(q_m)$ in $\bold C$, $q_m$ can be realized as a projection of $p_{n(m)}$
where the function $n(m)$ is bounded by a polynomial in $m$.
A sequence $(p_n)$ is {\it complete} for $\bold C$ if it is hard for $\bold C$ and if
$(p_n)\in \bold C$.
\end{definition}

A famous example of a   sequence in $\vp$  is $\tdet_n\in S^n\BC^{n^2}$, despite
its apparently huge expression size. While it is known that $(\tdet_n)\in \vp$, it
is not known whether or not it is $\vp$-complete. On the other hand, it is
known that $(\tdet_n)$ is $\vp_e$-hard, although it is not known whether or
not $(\tdet_n)\in \vp_e$. When complexity theorists and mathematicians are
confronted with such a situation, what else do they do other than
make another definition?

\begin{definition}\label{defvpws} 
The class $\vp_{ws}$ is the set of sequences $(p_n)$  where $\tdeg(p_n)$ is
bounded by a polynomial and such that
there exists a sequence of circuits $(C_n)$ of polynomialy bounded size  
such that $C_n$ represents $p_n$, and such that at any multiplication vertex, the component of
the circuit of one
of the two edges coming in is disconnected from the rest of the circuit by
removing the multiplication vertex.  
\end{definition}

In \cite{mapo:04} they show  $(\tdet_n)$ is $\vp_{ws}$-complete, so Conjecture \ref{valpermconj}
may be rephrased as conjecturing $\vp_{ws}\neq\vnp$.

\begin{remark} It is considered a major open question to determine whether or not
$(\tdet_n)\in \vp_e$.\end{remark}

\begin{definition} Given a complexity class $\bold C$ defined in terms of sequences of
polynomials, we define a sequence $(p_n)$ to be in $\ol{\bold C}$ if there exists
a curve of sequences $q_{n,t}$, such that for each fixed $t_0\neq 0$,
$(q_{n,t_0})\in \bold C$ and for all $n$, $\tlim_{t\ra 0}q_{n,t}=p_n$.
\end{definition}

\subsection{$\vnp$}
 
The class $\vnp$  essentially consists of polynomials whose coefficients
can be determined in polynomial time.
Consider a sequence $h=(h_n)\in \BC[x_1\hd x_n]_{\leq n}$ of 
(not necessarily homogeneous) polynomials of the form
\be\label{hneqn}
h_n=\sum_{e\in \{ 0,1\}^n}g_n(e)x_1^{e_1}\cdots x_n^{e_n}
\ene
where $(g_n)\in \vp$.  
Define $\vnp$ to be the set of all sequences
that are projections of sequences of the form $h$.
For equivalent definitions, see  e.g.,  \cite[\S 21.2]{BCS}.

\begin{proposition}\cite{MR564634} $(\tperm_n) \in \vnp$, in fact is $\vnp$-complete.
\end{proposition}

\begin{conjecture}\cite{vali:79-3}   [Valiant's hypothesis] $\vp\neq\vnp$.
\end{conjecture}

It is known that  $\p\neq \np$ would imply $\vp\neq \vnp$ over finite fields.

\section{Projecting the determinant to the permanent}\label{valsect}

\subsection{Complexity of $(\tdet_n)$}\label{gausselimsect}
For a vector space $V$, let $S^dV$ denote the space of homogeneous polynomials
of degree $d$ on the dual space $V^*$.
Let $E,F=\BC^n$, and let 
  $E\ot F$ denote  the space of linear maps $E^*\ra F$. The polynomial
$\tdet_n\in \La nE\ot \La nF\subset S^n(E\ot F)$ is the unique up to scale
(nonzero)
element of the one-dimensional vector space $\La nE\ot \La nF$.
  $\tdet_n$ is
invariant under the action of 
$SL(E)\times SL(F)$, as $\tdet(axb)=\tdet(a)\tdet(x)\tdet(b)$. Fix bases in $E,F$, so we may identify
$E\ot F$ with the space of $n\times n$ matrices and
$SL(E)$ as the subgroup of all $n\times n$ matrices with
determinant one.
If $x\in E^*\ot F^*$ is expressed
as a matrix, letting $\FS_n$ denote the permutation group on $n$ elements,  then
$$
\tdet_n(x)=\sum_{\s\in \FS_n} \sign (\s) x^1_{\s(1)}\hd x^n_{\s(n)}.
$$

  In the   na\"ive computation of $\tdet_n$ with this formula, one uses
$(n-1)(n!)$ multiplications and $n!-1$ additions.   Nevertheless, one has
the essentially classical:

\begin{proposition} \label{detinvp}$(\tdet_n)\in \vp$. More precisely,  $\tdet_n$ can be evaluated by performing $\cO(n^4)$
arithmetic operations.
\end{proposition}

Fixing bases of $E,F$ and identifying $E^*\ot F^*$ with the space of $n\times n$ matrices, there are
  subspaces   of  $E^*\ot F^*$ on which  $\tdet$ can be evaluated
by performing $n$ arithmetic operations, for example the 
upper-triangular matrices which we will denote by $\fb$.  

 $\tdet_n$ is invariant under the
action of  the subgroup $U\subset SL(E)$ of all upper-triangular
matrices with $1$'s on the diagonal as well as the group   $\cW$ of permutation matrices
in $SL(E)$.

  Proposition \ref{detinvp} essentially  follows from:

\begin{proposition}[Gaussian  elimination] Notations as above, given $x\in E^*\ot F^*$, there exists $g$ in the group generated
by $U$ and $\cW$ such 
that $g\cdot x\in \fb$. Such a $g$ can be computed by performing a number of
arithmetic operations that is polynomial in $n=\tdim E$.
\end{proposition}

\begin{proof}[proof of Prop. \ref{detinvp}] For sufficiently generic matrices the algorithm is clear and just using
$U$ is sufficient. For an algorithm
that works for arbitrary matrices, see, e.g.,  \cite{MR1479636,mapo:04}.
\end{proof}

\subsection{The permanent}
Define the permanent $\tperm_n\in S^n(E\ot F)$ to be the unique
up to scale
element of $S^nE\ot S^nF\subset S^n(E\ot F)$ invariant
under the action of the diagonal matrices and permutation
matrices acting on both the left and the right (i.e. the normalizers
of the tori in $SL(E)\times SL(F)$). 
If $x\in E^*\ot F^*$ is expressed
as a matrix, then
$$
\tperm_n(x)=\sum_{\s\in \FS_n} x^1_{\s(1)}\hd x^n_{\s(n)}.
$$

\subsection{The permanent as a projection of the determinant}

\begin{theorem}\label{pexpressthm}[Valiant]\cite{vali:82} Every $f\in \BC[x_1\hd x_n]$ of expression
size (see \S\ref{exprsize}) $u$ is both a projection of $\tdet_{u+3}$ and $\tperm_{u+3}$.
\end{theorem}

In particular,  any polynomial is the projection of some
determinant.

\begin{example}Let  $f(x)=x_1x_2x_3+x_4x_5x_6$, then
$$
f(x)=
\tdet
\begin{pmatrix}
0&x_1&0&x_4&0\\
0&1&x_2&0&0\\
x_3&0&1&0&0\\
0&0&0&1&x_5\\
x_6&0&0&0&1
\end{pmatrix}.
$$
\end{example}

\begin{conjecture}[Valiant]\label{valpermconj}\cite{MR564634}  Let $dc(\tperm_m)$ be the smallest integer
$n$ such that  $\tperm_m$ can  be realized
as a  projection of $\tdet_{n}$. Then    $dc(\tperm_m)$  grows faster
than any polynomial in $m$.  
\end{conjecture}

\subsection{Differential invariants of $\tdet_n$}
This subsection discusses preliminary results of work with D. The and L. Manivel.

Let $X\subset \pp n$ and $Y\subset\pp m$ be varieties such that there is a linear
space $L\simeq \pp m\subset \pp n$ such that $Y=X\cap L$.

Say   $y\in Y=X\cap L$. Then the differential invariants of
$X$ at $y$ will project to the differential invariants of $Y$ at $y$. A definition
of differential invariants adequate for this discussion   (assuming
$X$, $Y$ are hypersurfaces) is as follows:   choose local coordinates $(x^1\hd x^{n+1})$ for
$\pp n$ at $x=(0\hd 0)\in X$ such that $T_xX=\langle \frp{x^1}\hd \frp{x^n}\rangle$ and
expand out a Taylor series for $X$:
$$
x^{n+1}= r^2_{i,j}x^ix^j+ r^3_{i,j,k}x^ix^jx^k+\cdots
$$
The zero set   of $(r^2_{ij}dx^i\circ dx^j\hd r^k_{i_1\hd i_k}dx^{i_1}\ccdots dx^{i_k})$ 
in $\BP T_xX$ is
independent of choices.  I will refer to the polynomials
$F_{\ell,x}(X)$ although they are not well defined individually.
For more details see, e.g. \cite[Chap. 3]{IvL}.

One says that $X$ can approximate $Y$ to $k$-th order at $x\in X$ mapping to
$y\in Y$ if one can project the differential invariants to order
$k$ of $X$ at $x$ to those of $Y$ at $y$.

In \cite{MR2126826} it was shown that the determinant can approximate {\it any} polynomial
to second order if $n\geq \frac{m^2}2$ and that $\tperm_m$ is generic to order two, giving the lower
bound   $dc(\tperm_m)\geq \frac{m^2}2$. The previous lower bound  was $dc(\tperm_m)\geq \sqrt{2} m$ due to
J. Cai \cite{MR1032157} building on work of J.  von zur Gathen \cite{MR910987}.

One can ask what happens at higher orders. 

\medskip

If $X\subset \BP V$ is a quasi-homogeneous variety, i.e.,
a group $G$ acts linearly on 
$V$  and 
$X=\ol{G\cdot [v]}$ for some $[v]\in \BP V$, then
$T_{[v]}X$ is a $\fg([v])$-module, where $\fg([v])$ denotes the Lie algebra
of
the stabilizer of $[v]$ in $G$.

Let $e^1\hd e^n$ be a basis of $E^*$ and $f^1\hd f^n$ a basis of
$F^*$, let $v=e^1\ot f^1+\cdots + e^{n-1}\ot f^{n-1}$, so
$[v]\in \tzeros (\tdet_n)$ and $\tzeros(\tdet_n)=\ol{SL(E)\times SL(F)
\cdot [v]}$.

Write $E'=v(F)\subset E^*$, $F'=v(E)\subset F^*$ and set
$\ell_E=E^*/E'$, $\ell_F=F^*/F'$. Then, using $v$ to identify
$F'\simeq (E')^*$, one obtains 
$T_{[v]}\tzeros(\tdet_n)= \ell_E\ot F' \op (F')^*\ot F'\op (F')^*\ot \ell_F$
as a $\fg([v])-module$. Write an element of $T_{[v]}\tzeros(\tdet_n)$
as a triple $(x,A,y)$. In matrices,
$$
v=\begin{pmatrix} 1& & &  \\  & \ddots &  &  \\   & & 1&  \\   & & & 0\end{pmatrix},
\ \ \ \ 
T_{[v]}\sim \begin{pmatrix} A& y \\ x& 0\end{pmatrix}
$$

Taking the $\fg([v])$-module structure into account, it
is 
 straight-forward to show:

\begin{theorem}\label{detdiffinvars} Let $X=\tzeros(\tdet_n)\subset \pp{n^2-1}=\BP (E\ot F)$,   let $v=
 e_1\ot f_1+\cdots + e_{n-1}\ot f_{n-1} \in X$. With the notations above,
there exist bases in which  the differential invariants
of $X$ at $[v]$
are the polynomials
\begin{align*}
 F_{2,[v]}(X) &=  xy \\
 F_{3,x}(X) &=  xAy \\
&\vdots\\
  F_{k,x}(X) &=   xA^{k-2}y.
\end{align*}
\end{theorem}

Since the permanent hypersurface is not quasi-homogeneous, its differential invariants are more
difficult to calculate. It is even difficult to write down a general point in a nice way
(that depends on $m$, keeping in mind that we are not concerned with individual hypersurfaces,
but sequences of hypersurfaces). For example,  the point on the permanent
hypersurface chosen in  \cite{MR2126826} is not general as  there is a finite group that preserves
it. To get lower bounds it is sufficient to work with any point of
the permanent hypersurface, but one will not know if the obtained bounds are sharp.
To arrive at $dc(\tperm_m)$ being an exponential function of $m$, one might expect to improve
the exponent by one at each order of differentiation. The following theorem shows that
this does not happen at order three.

The Mignon-Ressayre result implies that
any hypersurface in $2n-2$ variables defined by a homogeneous polynomial  can be approximated to order two at any point by
an affine linear projection of $\{\tdet_n=0\}\subset \BC^{n^2}$.

\begin{theorem}\label{detprojs} 
Any hypersurface in $n-1$ variables can be approximated to order three at any point by
an affine linear projection of $\{\tdet_n=0\}\subset \BC^{n^2}$.

In particular, $\{ \tperm_m=0\}\subset \BC^{m^2}$ can be approximated to order  
  three  at a general point by an affine linear projection of 
  $\{\tdet_{m^2+1}=0\}\subset \BC^{(m^2+1)^2}$.
\end{theorem}
\begin{proof} 
The rank of $F_2$ for the determinant is $2(n-1)$, whereas the rank of $F_2$ for
the permanent, and   of a general  hypersurface  in $q$ variables at a general point,  is $q-2$.
so one would need to project
to eliminate $(n-1)^2$  variables to agree to order two.

Thus it is first necessary to  perform   a projection so that the matrix $A$, which has independent
variables as entries becomes  linear in the entries of $x,y$, write $A=A(x,y)$.
The projected pair $F_2,F_3$ is still not generic because it has two linear
spaces of dimension $n-1$ in its zero set. This can be fixed by setting $y=L(x)$ for
$L:\BC^{n-1}\ra \BC^{n-1}$ a linear isomorphism. 
At this point one has $F_2=L(x)x$, $F_3=L(x)A(x,L(x))x$. Take $L$ to be the identity map, so the
cubic is of the form 
$\sum_{i,j}x_iA_{ij}(x)x_j$ where the $A_{ij}(x)$ are arbitrary. This is an arbitrary cubic.
\end{proof}

\section{Geometric Complexity Theory approach to $\ol{\vp_{ws}}$ v.  $\vnp$}\label{MSsect}

In a series of   papers \cite{MS1,MS2,MS3,MS4,MS5,MS6,MS7,MS8},
 K. Mulmuley and M. Sohoni  outline an approach to prove $\ol{\vp_{ws}}\neq \vnp$.  
\medskip
Let  $\ell$ be a linear coordinate
on $\BC$, and   take  any linear inclusion $\BC\op \BC^{m^2}\subset \BC^{n^2}$ to have
$\ell^{n-m}\tperm_m$ be
a homogeneous degree $n$ polynomial  on  $\BC^{n^2}$. Mulmuley and Sohoni
observe that $\ol{\vp_{ws}}\neq \vnp$ is equivalent to the following assertion:
Let $\ol{dc}(\tperm_m)$ denote 
that the smallest value of $n$  such that 
$[\ell^{{  n}-m}\tperm_m]  \in \ol{GL_{{  n}^2} \cdot   [\tdet_{  n}]}$.
Then $\ol{\vp_{ws}}\neq \vnp$ is equivalent to the statement
$\ol{dc}(\tperm_m)$
grows faster than any polynomial:
\begin{conjecture}\cite{MS1}\label{msmainconj} $\ol{dc}(\tperm_m)$ grows faster than any  polynomial in $m$.
\end{conjecture}

\begin{remark}
Recently in \cite{LMRdet} it was shown that $\ol{dc}(\tperm_m)\geq\frac{m^2}2$ and that there
exist sequences $(p_m)$ with $\ol{dc}(p_m)<dc(p_m)$.
\end{remark}

\subsection{Description of the program to prove Conjecture \ref{msmainconj} outlined in \cite{MS2}}\label{gctdessect}

For a complex projective variety $X\subset \BP V$, let $I(X)\subset Sym(V^*)$ be the
ideal of polynomials vanishing on $X$.   Let $\BC[X] =Sym(V^*)/I(X)$ denote the homogeneous coordinate ring.
For complex projective
varieties $X,Y\subset \pp N=\BP V$, one has $X\subset Y$ iff $\BC[Y]$ surjects onto $\BC[X]$ (by restriction of
functions).    Mulmuley and Sohoni set out to prove:

\begin{conjecture}\cite{MS1}\label{msmainconjb} Let $u(m)$ be a polynomial. 
There is a sequence of irreducible modules $M_m$ for  $GL_{u(m)^2}$  such  that $M_m$ appears in
$\BC[\ol{GL_{u(m)^2} \cdot   [\ell^{u(m)-m}\tperm_m]}]$
but not in $\BC[\ol{GL_{u(m)^2} \cdot   [\tdet_{u(m)}]}]$.
\end{conjecture}

In an attempt to find such a sequence of modules,  Mulmuley and Sohoni  
 consider $SL_{n^2}\cdot \tdet_n$ and $SL_{m^2}\cdot \tperm_m$ because on the one hand their
coordinate rings can be determined in principle using representation theory, and on the other hand they are
closed affine varieties. They observe that any $SL_{n^2}$-module appearing in $\BC[SL_{n^2}\cdot \tdet_n]$ must
also appear in    $\BC[\ol{GL_{n^2}\cdot \tdet_n}]_k$ for some $k$.
Regarding the permanent, for $n>m$,  $SL_{n^2}\cdot\ell^{n-m}\tperm_m$ is not closed, so they develop
machinery to transport information about $\BC[SL_{m^2}\cdot \tperm_m]$ to $\BC[\ol{GL_{n^2}\cdot\ell^{n-m}\tperm_m}]$,
including  a notion of {\it partial stability}.

Mathematical aspects of this program are discussed in \cite{BLMW}.  
The representation-theoretic information Mumuley and Sohoni propose to exploit is
studied in detail. In particular \cite[Thm 5.7.1]{BLMW} is a precise description of conditions on {\it Kronecker
coefficients} that are equivalent to  Conjecture \ref{msmainconjb}. In addition, suggestions
are made for further geometric information that one could take into account that
might imply a more tractable problem in representation theory. 

The price of using $SL_{n^2}$ instead of $GL_{n^2}$ is that one loses the grading of
the coordinate rings. On the other hand, in order to use $GL_{n^2}$, one must
solve, or at least partially solve,  an {\it extension problem}, which to even
begin work on, means that one must determine the codimension one
components of the boundaries in the orbit closures.

\begin{remark} Recently in \cite{PBkron} evidence was given that the vanishing of Kronecker coefficients that would
be necessary for Conjecture \ref{msmainconjb}  
  is unlikely to occur.
\end{remark}

\subsection{Beyond determinant and permanent}\label{beyondgctsect}

Instead of considering $\tdet_n$, one could take a sufficiently generic $g_n\in GL_{n^2}$ and
consider $p_n:= \tdet_n+ g_n\cdot \tdet_n$. Then the subgroup $G(p_n)$ of $GL_{n^2}$ preserving $p_n$  will be the same as that for
a generic polynomial, although the sequence $(p_n)$ is still $\vp_{ws}$-complete. Thus just looking at
the orbit, there would   be {\it fewer} modules  appearing in $\BC[SL_{n^2}\cdot [p_n]]$ than  
in  $\BC[SL_{n^2}\cdot [\tperm_n]]$. In particular the orbit closure is larger than
that of the permanent. More generally, let $r(n)$ be a polynomial and
take a sequence of points in $p_n\in \s_{r(n)}(\ol{GL_{n^2}\cdot [\tdet_n]})$, the $r$-th
secant variety of $\ol{GL_{n^2}\cdot [\tdet_n]}$.
One could study the differential invariants of these varieties to see how they project
to the permanent as in \S\ref{valsect} and consider GCT program using
the varieties $\ol{GL_{n^2}\cdot p_n}$.

\medskip

More examples of sequences of polynomials are given by the {\it immanants} 
defined by Littlewood in \cite{MR2213154}.
Immanants generalize the determinant and permanent.  
 Given a partition $\pi=(p_1\hd p_r)$ of $n$,
and a vector space $V$ of dimension at least $r$, 
let $S_{\pi}V$ denote the corresponding irreducible $GL(V)$-module.
$IM_{\pi}\in S^n\BC^{n^2}$ may be defined as follows: consider
$\BC^{n^2}=E\ot F$, where $E,F=\BC^n$. Then
$S^n(E\ot F)=\op_{\pi}S_{\pi}E\ot S_{\pi}F$ as a $GL(E)\times GL(F)$ module.
Let $D^E\subset SL(E)$, $D^F\subset SL(F)$ denote  
the tori, i.e.,   the groups of diagonal matrices with determinant one.  Let $\FS_n^E,\FS_n^F$ denote
the groups of permutation matrices acting on the left and right, and let
$\Delta(\FS_n)\subset \FS_n^E\times \FS_n^F$ denote the diagonal embedding.
Then $IM_{\pi}\in S_{\pi}E\ot S_{\pi}F$ is the unique (up to scale) element
acted on trivially by $(D^E\times D^F)\ltimes \Delta(\FS_n)$.

In \cite{keye}, building on work in \cite{MR1275631,MR1412753}, it is shown that
for all non-self dual $\pi\neq (1^n),(n)$, that 
$G(IM_{\pi})= ((D^E\times D^F)\ltimes \Delta(\FS_n) )\ltimes \BZ_2$, where $\BZ_2$ acts
by sending a matrix to its transpose.

Consider
$IM_{(n-1,1)}$ and $IM_{(2,1^{n-1})}$. The first is $\vnp$-complete and the second is in $\vp$,
see \cite{buer:00-3}, so one  could attempt to apply the GCT program to them. By \cite{keye}
$G(IM_{(n-1,1)})=G(IM_{(2,1^{n-1})})$ so 
$\BC[SL_{n^2}\cdot [IM_{(n-1,1)}]]=\BC[SL_{n^2}\cdot [IM_{(2,1^{n-1})}]]$. 
Without examining the boundaries of 
$\ol{GL_{n^2}\cdot [IM_{(n-1,1)}]}$ and $\ol{GL_{n^2}\cdot [IM_{(2,1^{n-1})}]}$ there is
no way to distinguish them.

   Such investigations will be the subject of future work.

\section{Towards   geometric definitions of complexity classes}\label{geomccdefsect}

As mentioned several times, symmetry, sometimes in hidden form, appears to play a central role
in characterizing sequences in   $\vp$ that are apparently not in $\vp_e$. 
To 
  make a   geometric study of complexity, it would be desirable to  have coordinate free
definitions. In this section I give a coordinate free and geometric definition of the
class $\ol{\vp}_e$. I then give a coordinate free and geometric definition of
a class $\vp_{hs}$ which is intended as a first attempt    to geometrize the class $\vp$.
Unfortunately
at this writing I have no idea for a proposed purely geometric definition
of $\vnp$. (S. Basu and M. Shub, in separate personal communications,
have proposed that $\vnp$ should somehow be viewed as a bundle
over $\vp$, but I have been unable to make this precise.)

\subsection{Joins and multiplicative joins}

The {\it join} of projective varieties  $X_1\hd X_r\subset \BP V$, $J(X_1\hd X_r)\subset \BP V$,   is the Zariski closure of the points
of the form $[p_1+\cdots + p_r]$ with $[p_j]\in X_j$.
The expected dimension of $ J(X_1\hd X_r)$ is $\tmin(\sum\tdim X_j + r-1, \tdim \BP V)$.
Let $\hat T_{[p]}X\subset V$ denote the affine tangent space of $X$ at $[p]\in X$.  Terracini's lemma says that
if $([p_1]\hd[p_r])\in X_1\ctimes X_r$ is a general point, then
$$
\hat T_{[p_1+\cdots + p_r]}J(X_1\hd X_r)=\hat T_{[p_1]}X_1+\cdots + \hat T_{[p_r]}X_r.
$$

One can similarly define joins in affine space. The expressions are the same without the brackets.

\begin{definition}\label{mjdef}
Let $X\subset \BP S^aV$, $Y\subset \BP S^bV$ be varieties.
  Define the {\it multiplicative join} of $X$ and $Y$,   $MJ(X,Y)$,
by
$$
MJ(X,Y):=\{ [p  q] \mid [p]\in X,\ [q]\in Y\}\subset \BP S^{a+b}V.
$$
For varieties $X_j\subset \BP S^{d_j}V$,  define $MJ(X_1\hd X_r)\subset S^{d}V$
similarly (or inductively as $MJ(X,Y,Z)=MJ(X,MJ(Y,Z))$).
In the special case $X_j=\BP V\subset \BP S^1V$, $MJ(\BP V\hd \BP V)$ is the Chow variety 
of polynomials that decompose into a product of linear factors.
\end{definition}

Similarly, let $A_{d,v}$ denote the space of all polynomials of degree at most $d$ in $v$ variables.
For affine varieties $X\subset A_{d_1,v}$, $Y\subset A_{d_2,v}$, $MJ(X,Y)\subset A_{d_1+d_2,v}$ 
is defined in the same way without brackets.

\begin{proposition}Let $X_j\subset \BP S^{d_j}V$ be varieties and let $([p_1]\hd [p_r])\in X_1\ctimes X_r$ be a general point. Then
$$\hat T_{[p_1\ccdots p_r]}MJ(X_1\hd X_r)=
\hat T_{[p_1]}X_1\circ p_2\circ\cdots\circ p_r+
\cdots + p_1\circ\cdots\circ p_{r-1}\circ \hat T_{[p_r]}X_r
$$ 
In particular, the expected dimension of  $MJ(X_1\hd X_r)$ is $\tmin(\tdim X_1+\cdots + \tdim X_r, \tdim \BP S^dV)$.
\end{proposition}

\begin{proof} Let $p_j(t)$ be a curve in $X_j$ with $p_j(0)=p_j$.
Differentiate the expression $p_1(t)\circ\cdots\circ p_r(t)$ at $t=0$ to get the result.
\end{proof}

\begin{question} What are the degenerate multiplicative joins, i.e., those that fail to be of the expected dimension?
\end{question}

\subsection{A geometric  characterization of $\ol{\vp}_e$}

Recall that the expression size $E(p)$ of a polynomial $p\in A_{d,v}$ is given by
the number of internal nodes of the smallest tree circuit computing $p$.
Define $\ol{E}(p)$ to be the smallest integer such that
there is a curve $p_t$ with $\tlim_{t\ra 0}p_t=p_0$ and such
that $E(p_t)=\ol{E}(p)$ for $t\neq 0$.
By definition, a
  sequence  $(p_n)\in A_{d(n),v(n)}$ is in $\vp_e$ (resp. $\ol{\vp}_e$)  if there exists
a polynomial $r(n)$  such that $E(p_n)\leq r(n)$ (resp. $\ol{E}(p_n)\leq r(n)$.).

\bigskip

To  a tree circuit $\G$ associated to a polynomial $p$,    associate an algebraic variety
as follows: first form an new tree circuit $\G'$ by   collapsing
all pairs of input nodes that are joined by a $+$ to a single input node, and   repeat 
as many times as necessary until no pairs of input nodes are joined by a $+$.
  (I take this first step to eliminate
the choice of coordinates involved in making the circuit.)
Associate to each input node a copy of $\BP V$.

Thus on $\G'$, if any two input nodes are joined, they are joined by a $*$-node.
Now perform a step by step procedure to eliminate all $*$-nodes joining
pairs of input nodes.  
Take  a $*$-node joined to two input nodes, and form a subtree containing  all other $*$-nodes joined to it and
an input node. Say there
are $j_1-1$ such. Record the variety $MJ_{j_1}:=MJ(\BP V\hd \BP V)$ of $j_1$ copies of $\BP V$. Collapse the subtree to
a single input node and associate $MJ_{j_1}$ to this input node. Now start again, say we arrive
at $j_2-1$ nodes in the subtree and record the variety $MJ_{j_2}=MJ(\BP V\hd \BP V)$ of $j_2$ copies of $\BP V$.
Continue until we have recorded $p$ varieties of multiplicative joins of $V$ of various
sizes. 

We arrive at a new graph $\G''$ all of whose $p$ input nodes
have varieties $MJ_{j_i}$ associated to them and when input nodes    are paired together by an
internal node, the node is a   $+$-node. 
Now   perform a step by step procedure to eliminate all $+$'s joining
pairs of input nodes.  
Take the first $+$, say  that the variety $MJ_{j_{i_1}}$ 
is one of the input nodes and    form a subtree consisting of all other $+$'s joined to it. Say there
are $k-1$ such. Record the variety $J(MJ_{j_{i_1}}\hd MJ_{j_{i_{k}}})$. Collapse the subtree to
a single input node and associate $J(MJ_{j_{i_1}}\hd MJ_{j_{i_{k}}})$ to this input node.  
Continue until we have  varieties of   joins of multiplicative joins  of various
sizes as our new input nodes with all pairings of input nodes $*$-nodes. 

Now continue as we did with $\G'$, taking multiplicative joins (of the joins of multiplicative joins)
until the further collapsed graph has all pairings of input nodes $+$'s, then go back to taking
joins etc...

This process terminates after a number of steps fewer than the number of nodes of $\G$, and one
arrives at a variety $\Sigma_{\Gamma}$
of successive joins and multiplicative joins.
  By construction $p \in \Sigma_{\G}$. 

Note that for each such variety, there are many $\G$ that are associated to it, but each has, up to the 
initial $v$ times the number of initial input nodes, the same expression size.

Let $\Sigma_R^{d,v}$ denote the union of all the varieties obtainable from a graph of at most $R$ internal  nodes
computing an element of $A_{d,v}$.
There is a finite number of such, so $\Sigma_R^{d,v}$ is an algebraic variety.
The above discussion implies

\begin{theorem}\label{olvpechar}  Let $p_n\in A_{d(n),v(n)}$ be a sequence with $d,v$ polynomials. Then $(p_n)\in\ol{\vp}_e$ iff there exists
a polynomial $R(n)$ and $p_n\in \Sigma^{d(n),v(n)}_{R(n)}$. In other words the complexity class
$\ol{\vp}_e$ is characterized by a sequence of algebraic varieties.
\end{theorem}

\begin{remark}  One has to use the class $\ol{\vp}_e$ instead of $\vp_e$ because when taking
joins one must include limits. It is not necessary to include limits when taking
multiplicative joins.
\end{remark}

\begin{corollary}A sequence $(p_n)\in A_{d(n),v(n)}$ is in $\ol{\vp}_e$ if
either $d$ or $v$ is constant. A generic sequence in $A_{d(n),v(n)}$ is
not in $\ol{\vp}_e$ if both $d,v$ grow at least linearly with respect to $n$.
\end{corollary}

\begin{proof} $\tdim \Sigma^{d,v}_R\leq (v+1)(R+1)$.\end{proof}

\subsection{Towards a geometric understanding of $\vp$}
Recall that  the
determinant  has the property that for each $n$ there is a subspace $\fb_n\subset \BC^{n^2}$, such that $\tdet_n\mid_{\fb_n}\in \vp_e$ and moreover $G(\tdet_n)\cdot \fb_n=\BC^{n^2}$.
This perspective motivates the following definitions.

Define  $\vp^{prim}$  to be the set of sequences   $p_n\in A_{d(n),v(n)}$, where for
each $n$, there exists a linear subspace $\Sigma_n\subset \BC^{v(n)}$, such that the sequence $(p_n)|_{\Sigma_n}$ lies in $\vp_e$, and
letting $G(n)$ denote the subgroup of $GL_{v(n)}$ preserving $(p_n)$, ask moreover that $G(n)\cdot \Sigma_n=\BC^{v(n)}$.
Clearly $\vp^{prim}\subset \vp$  as the action of $G(n)$ is cheap.
 $\vp^{prim}$  is modeled on $(\tdet_n)$ where $\Sigma_n$ is the upper-triangular matrices.
Define $\vp_{hs}$ to be  set of sequences $(p_n)$ such that there exists another
sequence $(r_n)$ with $(r_n)\in \vp_e$, a polynomial $q(n)$, and sequences $(p_{n,j})$, $j=1\hd q(n)$ such that
$(p_{n,j})\in \vp^{prim}$ and $p(n)=r_n(p_{n,1}\hd p_{n,q(n)})$.
Then $\vp_{hs}\subseteq \vp$.

\begin{question}  What is the gap, if any,  between $\vp$ and $\vp_{hs}$?
\end{question}

\bibliographystyle{amsplain}
 
\bibliography{Lmatrix}

\end{document}